\documentclass[1p,preprint]{elsarticle}

\usepackage{graphicx, amssymb, latexsym, amsfonts, amsmath, lscape, amscd,
amsthm, color, epsfig, mathrsfs,tikz,subfigure}
\usetikzlibrary{calc}
\usetikzlibrary{decorations.markings,arrows}
\newtheorem{theorem}{Theorem}[section]
\newtheorem{observation}[theorem]{Observation}
\newtheorem{conjecture}[theorem]{Conjecture}
\newtheorem{corollary}[theorem]{Corollary}

\newtheorem{lemma}[theorem]{Lemma}

\newtheorem{claim}{Claim}

\newfont{\Bb}{msbm10 scaled\magstep1}
\newcommand{\Cay}{\mbox{Cay}}

\begin{document}

\title{Homomorphisms of binary Cayley graphs}%

\author[LIMOS]{Laurent Beaudou}
\author[LRI]{Reza Naserasr}
\author[RMC]{Claude Tardif}

\address[LIMOS]{CNRS, LIMOS, UMR6158, Univ. Clermont-Ferrand 2, Aubi\`{e}re -- France}
\address[LRI]{CNRS, LRI, UMR8623, Univ. Paris-Sud 11, 
F-91405 Orsay Cedex -- France}
\address[RMC]{Coll\`{e}ge Militaire Royal du Canada, Kingston, Ontario -- Canada}

\begin{abstract}
  A binary Cayley graph is a Cayley graph based on a binary group. In
  1982, Payan proved that any non-bipartite binary Cayley graph must
  contain a generalized Mycielski graph of an odd-cycle, implying that such a 
  graph cannot have chromatic number 3.

  We strengthen this result first by proving that any non-bipartite
  binary Cayley graph must contain a projective cube as a subgraph. We
  further conjecture that any homomorphism of a non-bipartite binary
  Cayley graph to a projective cube must be surjective and we prove some 
  special case of this conjecture.
\end{abstract}


\maketitle

\section{Introduction}

For classic notation we will follow that of \cite{GodsilRoyle}.  A
\emph{binary Cayley} graph is a Cayley graph $\Cay(\Gamma, \Omega)$
where $\Gamma$ is a binary group (i.e., $x+x=0$ for any element $x$),
and $\Omega$ is any subset of $\Gamma$ (normally not including element
$0$). The vertices of the graph are the elements of $\Gamma$, and two
vertices $u$ and $v$ are adjacent if and only if $u-v \in
\Omega$. Thus $\Cay(\Gamma, \Omega)$ is a simple graph when element
$0$ is not in $\Omega$. Hypercubes are the most famous examples of
binary Cayley graphs. In fact, for this reason, binary Cayley graphs
often are referred to as \emph{cube-like graphs}.


Other examples of binary Cayley graphs, which are essential for this
work, are the \emph{projective cubes}. A projective cube of dimension
$d$, denoted $\mathcal{PC}_d$, is defined as the Cayley graph
$\Cay(\mathbb{Z}_2^d, \{e_1,e_2,\cdots, e_d, J\})$ where $(e_1, e_2,
\ldots, e_d)$ is the canonical basis and $J$ is the all-1
vector. Projective cube of dimension $d$ can be built from hypercube
of dimension $d+1$ by identifying antipodal vertices. From this fact
comes their name. It can also be built, equivalently, from the
hypercube of dimension $d$ by adding edges between antipodal pairs of
vertices. This satisfies the Cayley graph definition given here.  In
some literature they are also referred to as \emph{folded cubes}.
Projective cubes are studied for their highly symmetric structures.
Homomorphisms to projective cubes capture some important packing and
edge-coloring problems, see~\cite{N07, NRS13}.

A graph $G$ is a {\em core} if it does not admit a homomorphism to a
proper subgraph of itself.

In this work we show the importance of projective cubes in the study
of homomorphisms of Cayley graphs on binary groups. Among other
properties, we will need the following results:

\begin{theorem}[Naserasr 2007 \cite{N07}]
 The projective cube of dimension $2k-1$ is bipartite. 
 Projective cube of dimension $2k$ is of odd girth $2k+1$. Furthermore,
 any pair of vertices of $\mathcal{PC}_{2k}$ is in a common cycle of length 
$2k+1$. 
\end{theorem}

\begin{corollary}\label{ProjectiveCubeCORE}
 The projective cube of dimension $2k$ is a core.
\end{corollary}

In \cite{Payan98}, Payan proved a surprising result that there is no
binary Cayley graph of chromatic number 3. His proof was an
implication of the following stronger result based on the following
definition.  Let $G$ be a graph on vertices $v^0_1, v^0_2, \ldots,
v^0_n$. The $k$-th level Mycielski graph of $G$, denoted $M^k(G)$, is
built from $G$ by adding vertices $v^1_1, v^1_2, \ldots,v^1_n$,
$v^2_1, v^2_2, \ldots, v^2_n$ up to $v^k_1, v^k_2, \ldots, v^k_n$
where if $v^0_i$ is adjacent to $v^0_j$, then $v^r_i$ is also adjacent
to $v^{r-1}_j$, finally we add one more vertex $w$ which is is joined
to all vertices $v^k_i$.  We will use the following result of
Stiebitz, see~\cite{Matousek2003} for a proof.

\begin{lemma}[Siebitz 1985 \cite{Stiebitz85}]
Let $C$ be an odd-cycle.
Then for any $i$, $\chi(M^i(C)) = 4$. 
\end{lemma}

Payan proved the following stronger statement:

\begin{theorem}[Payan 1998 \cite{Payan98}] 
\label{Payan}
 Given a binary Cayley graph $\Cay(\Gamma, \Omega)$ of odd-girth
 $2k+1$, the $k$-th level Mycielski graph $M^k(C_{2k+1})$ is a
 subgraph of $\Cay(\Gamma, \Omega)$.
\end{theorem}

This in particular implies that the projective cube of dimension $2k$
contains the graph $M^{k}(C_{2k+1})$ as a subgraph. This fact is also
implied from the following view of the projective cubes.

First, recall that for any pair of integer $n,k$ with $k < n$, the
graph $K(n,k)$ is the {\em Kneser graph} of $k$ among $n$. Its vertex
set is made by the $n \choose k$ subsets of $[1 \cdot n]$ of size $k$,
two of them being adjacent if they are disjoint.

Now, for an integer $k$, and a set $\mathcal A$ of size
$2k+1$. Vertices of $\mathcal{PC}_{2k}$ can be regarded as the partitions $(A,
\bar A)$ of $\mathcal A$. We always assume $A$ is the smaller
part. Two such vertices $(A, \bar A)$ and $(B, \Bar B)$ are adjacent
if either $A$ or $\bar A$ is obtained from $B$ by adding one more
element. This implies that the subgraph induced by vertices $(A, \bar
A)$ with $|A|=k$ is isomorphic to the Kneser graph $K(2k+1, k)$. To
find $M^k(C_{2k+1})$ in this graph, just take $v^0_1, v^0_2, \ldots
v^0_{2k+1}$ to be a $2k+1$-cycle in this Kneser graph. Call
$(A_i,\bar{A_i})$ the partition associated with $v^0_i$. Then for each $j$,
$A_{j-1}$ and $A_{j+1}$ (indices are taken modulo $2k+1$) have
exactly $k-1$ elements in common. Let $A^1_j$ be this subset and
define $v^1_j$ to be $(A^1_j \bar A^1_j)$. Continuing by induction each
pair $v^i_{j-1}$ and $v^i_{j+1}$ of vertices define a unique set of
size $i-1$ which defines $v^{i+1}_{j}$ with the last vertex being
$(\emptyset, \mathcal{A})$.

In Section \ref{sec:power}, we strengthen the result of Payan proving
that:

\begin{theorem}\label{PCasSubgraph}
 Given a binary Cayley graph $\Cay(\Gamma, \Omega)$ of odd-girth
 $2k+1$, the projective cube $\mathcal{PC}_{2k}$ is a subgraph of
 $(\Gamma, \Omega)$.
\end{theorem}

Since a $k$-coloring of a graph $G$ is equivalent to a homomorphism of
$G$ to $K_k$, the corollary of Payan's theorem can be restated as
follows:

\begin{theorem}[Payan 1998 \cite{Payan98}]
  \label{PayanCorollary}
 If a non-bipartite binary Cayley graph admits a homomorphism to
 $K_4$, then any such homomorphism must be a surjective mapping.
\end{theorem}

Considering the fact that $K_4$ is isomorphic to $\mathcal{PC}_{2}$,
we introduce the following conjecture in generalization of
Theorem~\ref{PayanCorollary}.

\begin{conjecture}
  \label{MappingToPC2k}
 If a non-bipartite binary Cayley graph admits a homomorphism to
 $\mathcal{PC}_{2k}$, then any such homomorphism must be an onto
 mapping.
\end{conjecture}

In Section \ref{sec:bintopc}, we reduce this conjecture to properties
of homomorphisms among Projective cubes only. Then we prove a special
case.


\section{Power graphs and pseudo-duality}
\label{sec:power}

Given a set $A$, the {\em power set} of $A$ is the set of all subsets
of $A$. It is denoted by ${\cal P }(A)$. This set forms a binary group
together with the operation of \emph{symmetric difference}. In fact it
is isomorphic to $(\mathbb{Z}_{2}^{|A|},+)$, each subset being represented
by its characteristic vector.

For a graph $G$, let $\widehat{G}$ denote the Cayley graph $\Cay(
{\cal P}(V(G)), E(G))$. This is the graph whose vertices are the subsets
of vertices of $G$ where two vertices are adjacent if their symmetric
difference is an edge of $G$.  It is worth noting that $E(G)$ is the
smallest Cayley subset which makes the natural injection of $G$ into
$\widehat{G}$ a homomorphism. Recall that a homomorphism is an edge
preserving mapping of vertices.

The graph $P_n$ is the path on $n$ vertices. The power graph
$\widehat{P_n}$ consists of two connected components each isomorphic
to the hypercube of dimension $n-1$. For a cycle, $C_n$, the power
graph $\widehat{C_n}$ consists of two connected components each
isomorphic to the projective cube of dimension $n-1$.

In general the following holds.
\begin{lemma} \label{1}
For a graph $G$, an integer $n$ and a Cayley graph $H$ on ${\mathbb
  Z}_2^n$, there exists a homomorphism from $G$ to $H$ if and only if
there exists a homomorphism from $\widehat{G}$ to $H$.
%
\end{lemma}



We will prove Lemma~\ref{1} in a much more general form,
encompassing all varieties of groups. Let ${\cal V}$ be
a variety of groups, that is, a class of groups defined
by a set of equations. For instance the variety of abelian groups
is defined by the equation $xy = yx$, and the groups
$\mbox{\Bb Z}_2^n$ are (up to isomorphism) the finite
members of the variety of groups defined by the equation
$x^2 = 1$.

For a graph $G$, we denote by ${\cal F_V}(G)$ the free group
on the vertex set of $G$ in the variety ${\cal V}$, and 
$S_{\cal V}(G)$ the following subset of ${\cal F_V}(G)$:
$$
S_{\cal V}(G) = \{ u^{-1}v : \{u,v\} \in E(G) \}.
$$
The general form of Lemma~\ref{1} is the following.
\begin{lemma} \label{1+}
Let $\Cay(A,S)$ be a Cayley graph, where
$A$ is a group in ${\cal V}$. Then for a graph $G$,
there exists a homomorphism of $G$ to $\Cay(A,S)$
if and only if there exists 
a homomorphism of $\Cay({\cal F_V}(G),S_{\cal V}(G))$ to $\Cay(A,S)$.
\end{lemma}

\begin{proof}
By definition of $S_{\cal V}(G)$, the inclusion of $V(G)$ in ${\cal
  F_V}(G)$ gives a natural homomorphism from $G$ to $\Cay({\cal
  F_V}(G),S_{\cal V}(G))$.  Therefore, if there exists a homomorphism
of $\Cay({\cal F_V}(G),S_{\cal V}(G))$ to $\Cay(A,S)$, then there
exists a homomorphism from $G$ to $\Cay(A,S)$.

Now suppose that there exists a graph homomorphism $\phi: G \rightarrow
\Cay(A,S)$.  Then $\phi$ extends to a group homomorphism
$\widehat{\phi}: {\cal F_V}(G) \rightarrow A$, and it is easy to see that
$\widehat{\phi}$ is also a graph homomorphism of $\Cay({\cal
  F_V}(G),S_{\cal V}(G))$ to $\Cay(A,S)$.  Indeed, if the set
$\{w_1,w_2\}$ is an edge in $\Cay({\cal F_V}(G),S_{\cal V}(G))$, then
$w_1^{-1}w_2 = u^{-1}v$ for some $\{u,v\} \in E(G)$, whence
$\widehat{\phi}(w_1)^{-1}\widehat{\phi}(w_2) = \phi(u)^{-1}\phi(v)$ which is in $S$.
\end{proof}

Note that when ${\cal V}$ is the variety of all groups, then ${\cal
  F_V}(G)$ is simply the free group on $V(G)$, and Lemma~\ref{1+}
presents $\Cay({\cal F_V}(G),S_{\cal V}(G))$ as the smallest Cayley
graph into which $G$ admits a homomorphism.  By a result of Sabidussi
\cite{sabidussi64} reformulated in~\cite{HahnTardif96}, every
vertex-transitive graph is a retract of a Cayley graph. Therefore
$\Cay({\cal F_V}(G),S_{\cal V}(G))$ is also the smallest
vertex-transitive graph into which $G$ admits a homomorphism.  In
particular, the chromatic number of $\Cay({\cal F_V}(G),S_{\cal
  V}(G))$ is equal to that of $G$; since the chromatic number is
defined in terms of homomorphisms into complete graphs, which are
Cayley graphs.  The fractional chromatic number of $G$ is defined in
terms of homomorphisms to Kneser graphs (see
\cite{ScheinermanUllman97}), which are seldom Cayley graphs (see
\cite{Scapellato96}) but nonetheless vertex-transitive; therefore the
fractional chromatic number of $\Cay({\cal F_V}(G),S_{\cal V}(G))$ is
equal to that of $G$.

When ${\cal V}$ is the variety of abelian groups, then the chromatic
number of the Cayley graph $\Cay({\cal F_V}(G),S_{\cal V}(G))$ is
again equal to that of $G$, since the complete graphs are also Cayley
graphs on abelian groups. However the fractional chromatic number of
$\Cay({\cal F_V}(G),S_{\cal V}(G))$ may be larger than that of $G$.
For instance, it can be shown that the fractional chromatic number of
the Petersen graph $P$ is $\frac{5}{2}$, while that of $\Cay({\cal
  F_V}(P),S_{\cal V}(P))$ is $3$.

Now, the finite groups in the variety ${\cal V}$ defined by the
identity $x^2 = 1$ are all isomorphic to $\mbox{\Bb Z}_2^n$ for some
$n$.  Therefore only the complete graphs whose number of vertices is a
power of $2$ are Cayley graphs on groups in ${\cal V}$, so for an
arbitrary graph $G$, even the chromatic number of $\Cay({\cal
  F_V}(G),S_{\cal V}(G))$ ( which is equal to $\widehat{G}$) may be
larger than that of $G$. In essence, Corollary~\ref{PayanCorollary}
goes a step further than this observation, by stating that the number
$3$ does not even belong to the range of chromatic numbers of Cayley
graphs of groups in ${\cal V}$.

Note that $\widehat{C_n}$ consists in two disjoint copies of
$\mathcal{PC}_{n-1}$. Thus if $C_{2k+1}$ maps to a binary Cayley graph
$G$, then, by Lemma~\ref{1}, the projective cube $\mathcal{PC}_{2k}$
maps to $G$. Furthermore, if $2k+1$ is the length of the shortest
odd-cycle of $G$, then in any mapping of $\mathcal{PC}_{2k}$ to $G$ no
two vertices of $\mathcal{PC}_{2k}$ can be identified. This proves the claim of
Theorem~\ref{PCasSubgraph}.
%


\section{Mapping binary Cayley graphs to projective cubes}
\label{sec:bintopc}

By restating Payan's theorem with the language of homomorphisms, we
obtain Theorem~\ref{PayanCorollary}. This led us to formulate
Conjecture~\ref{MappingToPC2k}, suggesting that what makes 4-coloring
so special is the fact that $\mathcal{PC}_2$ is isomorphic to $K_4$.






In the context of this conjecture, note that since $G$ is not
bipartite it contains an odd-cycle. Let $2r+1$ be the length of a
shortest odd-cycle of $G$. Since $G$ maps to $\mathcal{PC}_{2k}$ and
since the odd-girth of $\mathcal{PC}_{2k}$ is $2k+1$, we have $r\geq
k$. On the other hand Theorem~\ref{PCasSubgraph} tells us that $G$
contains $\mathcal{PC}_{2r}$ as a subgraph. Since $\mathcal{PC}_{2r}$
itself is a binary Cayley graph, Conjecture~\ref{MappingToPC2k} is
equivalent to the following conjecture.

\begin{conjecture}\label{MappingAmongPC2k}
Given $r\geq k$, any mapping of $\mathcal{PC}_{2r}$ to
$\mathcal{PC}_{2k}$ must be onto.
\end{conjecture}

When $k$ is equal to 1, this conjecture is equivalent to Payan's
theorem and is implied by the fact that $M^k(C_{2k+1})$ is a subgraph
of $\mathcal{PC}_{2k}$ as mentioned in the introduction.  The case
when $k$ is equal to $r$ is also equivalent to stating that
$\mathcal{PC}_{2k}$ is a core as observed by Corollary
\ref{ProjectiveCubeCORE}.  In the next theorem we verify the
conjecture for $k=2$ and $r=3$.  In other words we prove that any
homomorphism of $\mathcal{PC}_6$ into $\mathcal{PC}_4$ should be
surjective. We start with a couple of observations that might be
useful in general case.

\begin{observation}\label{obs:dist2}
 If $f: \mathcal{PC}_{2k+2} \rightarrow \mathcal{PC}_{2k}$ is a
 homomorphism and $f(x)=f(y)$, then $x$ and $y$ have a common
 neighbor, i.e., they are at distance 2.
\end{observation}

\begin{proof}
  Vertices $x$ and $y$ belong to a cycle of length $2k+3$ in
  $\mathcal{PC}_{2k+2}$. If they are not at distance 2, then there
  would be a cycle of odd length strictly smaller than $2k+1$ in
  $\mathcal{PC}_{2k}$ which is a contradiction.
\end{proof}

\begin{figure}
  \begin{center}
    \begin{tikzpicture}[line cap=round,line join=round,x=.5cm,y=.5cm]
      \draw (0,0) node [above] {$\emptyset$};
      \fill (0,0) circle (1.5pt);
      \draw (-6,-3) node [left] {$1$};
      \draw (-3,-3) node [left] {$2$};
      \foreach \i in {3,4,...,5}{
        \draw (3*\i - 9,-3) node [right] {$\i$};
      }
      \foreach \i in {1,2,...,5}{
        \fill (3*\i - 9,-3) circle (1.5pt);
        \draw (0,0) -- (3*\i - 9,-3);
        \foreach \j in {1,2,...,4}{
          \draw [dotted] (3*\i - 9,-3) -- (.5*\j + 3*\i - 10.25, -4.5);
        }
      }
      
      \coordinate (centre) at (0,-10.5);
      \foreach \angle in {90,162,18,234,306}{
        \coordinate (a) at ($ (centre) + (\angle:2)$);
        \coordinate (b) at ($ (centre) + (\angle-144:2)$);
        \coordinate (c) at ($ (centre) + (\angle-144:4)$);
        \coordinate (d) at ($ (centre) + (\angle-72:4)$);
        \coordinate (e) at ($ (b) + (\angle - 154:1)$);
        \coordinate (f) at ($ (b) + (\angle - 134:1)$);
        \coordinate (g) at ($ (c) + (\angle - 154:1)$);
        \coordinate (h) at ($ (c) + (\angle - 134:1)$);
        \draw (a) -- (b) -- (c) -- (d);
        \draw [dotted] (b) -- (e);
        \draw [dotted] (b) -- (f);
        \draw [dotted] (c) -- (g);
        \draw [dotted] (c) -- (h);
        \fill (b) circle (1.5pt);
        \fill (c) circle (1.5pt);
      }
      \draw ($ (centre) + (18:2)$) node[below] {$24$}; 
      \draw ($ (centre) + (18:4)$) node[above] {$35$}; 
      \draw ($ (centre) + (90:2)$) node[right] {$34$}; 
      \draw ($ (centre) + (90:4)$) node[right] {$12$}; 
      \draw ($ (centre) + (162:2)$) node[below] {$13$}; 
      \draw ($ (centre) + (162:4)$) node[above] {$45$}; 
      \draw ($ (centre) + (234:2)$) node[right] {$15$}; 
      \draw ($ (centre) + (234:4)$) node[left] {$23$}; 
      \draw ($ (centre) + (306:2)$) node[left] {$25$}; 
      \draw ($ (centre) + (306:4)$) node[right] {$14$}; 
      
    \end{tikzpicture}
  \end{center}
  \caption{A depiction of $\mathcal{PC}_{4}$.}
  \label{fig:pc4}
\end{figure}
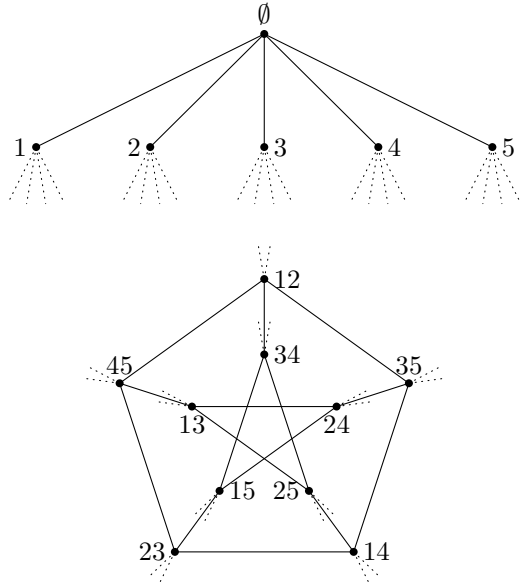

\begin{corollary}\label{preImageOf5}
If $f: \mathcal{PC}_{2k+2} \rightarrow \mathcal{PC}_{2k}$ is a homomorphism and
$|f^{-1}(x)|\geq 5$ for some vertex $x\in V(\mathcal{PC}_{2k})$, then
$f^{-1}(x) \subseteq N(a)$ for some $a\in V(\mathcal{PC}_{2k+2})$.
\end{corollary}

\begin{proof}
 Using the poset notation, and without loss of generality we may
 assume that the vertex associated with the empty set is in
 $f^{-1}(x)$.  Then every other vertex in $f^{-1}(x)$ must be a
 2-subset of $[1 \cdot 2k+3]$. Moreover they must be at distance 2
 from each other, so that each pair of 2-subsets in $f^{-1}(x)$ have a
 non-empty intersection. In order to reach four such 2-subsets, there
 has to be a fixed element (say $i$) in all of them. Let $a$ be the
 vertex associated with the set $\{i\}$ in $\mathcal{PC}_{2k+2}$, we
 then have $f^{-1}(x) \subseteq N(a)$.
\end{proof}

\begin{observation}\label{obs:noPreimageOf6}
  If there exists a homomorphism of $\mathcal{PC}_{6}$ into
  $\mathcal{PC}_{4}$ which is not surjective and such that a vertex of
  $\mathcal{PC}_{4}$ has a pre-image of size $6$, then there exists a
  homomorphism of $\mathcal{PC}_{6}$ into $\mathcal{PC}_{4}$ which is
  not surjective and with no vertex of $\mathcal{PC}_{4}$ being the
  image of $6$ vertices of $\mathcal{PC}_{6}$.
\end{observation}

\begin{proof}
  Let $f$ be a homomorphism of $\mathcal{PC}_{6}$ into
  $\mathcal{PC}_{4}$ which is not surjective and such that a vertex
  $x$ of $\mathcal{PC}_{4}$ has a pre-image of size $6$. By Corollary
  \ref{preImageOf5}, there exists a vertex $a$ of $\mathcal{PC}_{6}$
  such that $f^{-1}(x) \subseteq N(a)$. Without loss of generality, we
  may assume that $a$ is the vertex associated with the empty set and
  that $f^{-1}(x)$ is made of the singletons from $\{1\}$ to
  $\{6\}$. Let $y$ be the image of the singleton $\{7\}$. It cannot be
  the image of $7$ vertices (otherwise it would be the whole
  neighborhood of a vertex in $\mathcal{PC}_{6}$ but each of the
  neighbors of $\{7\}$ has one of its neighbors mapped to
  $x$). Therefore, mapping the singleton $\{7\}$ to $x$ does not
  create a new vertex of $\mathcal{PC}_{4}$ being the image of $6$
  vertices of $\mathcal{PC}_{6}$. One can easily check that it is
  still a homomorphism and it remains not surjective. We thus have
  built a homomorphism from $\mathcal{PC}_{6}$ to $\mathcal{PC}_{4}$
  which is not surjective and with strictly less vertices of
  $\mathcal{PC}_{4}$ being the image of exactly $6$ vertices of
  $\mathcal{PC}_{6}$. We may keep doing so until there is no such
  vertex.
\end{proof}

\begin{observation}\label{obs:5gives5}
  Let $f$ be homomorphism of $\mathcal{PC}_{6}$ into
  $\mathcal{PC}_{4}$. If there is a vertex $x$ of $\mathcal{PC}_{4}$
  with a pre-image of size 5 or more, then there is a vertex $y$
  adjacent to $x$ with a pre-image of size 5 or more. Moreover the
  common neighbor of the vertices in the pre-image of $y$ is adjacent
  to the common neighbor of the vertices in the pre-image of $x$.
\end{observation}

\begin{proof}
  With Corollary \ref{preImageOf5}, we may assume that $f^{-1}(x) =
  \left\{ \{1\}, \{2\}, \{3\}, \{4\}, \{5\} \right\}$, the empty set
  being the common neighbor of the pre-image of $x$. This last set
  has twenty-one neighbors in $\mathcal{PC}_{6}$ that must be mapped to the five
  neighbors of $x$. One of these neighbors of $x$, must have a
  pre-image of size 5 or more. Let it be $y$. The only vertices having
  more than five neighbors in $N(f^{-1}(x))$ are the vertices
  associated with singletons. Therefore the common neighbors to the
  vertices of the pre-image of $y$ is a singleton which is adjacent to
  the empty set.
\end{proof}

\begin{theorem}\label{thm:main}
 Any homomorphism of $\mathcal{PC}_{6}$ into $\mathcal{PC}_{4}$ must be onto.
\end{theorem}

\begin{proof} 
For a contradiction, let $f: \mathcal{PC}_{6} \rightarrow
\mathcal{PC}_{4}$ be a homomorphism which is not onto. By Observation
\ref{obs:noPreimageOf6}, we may assume that for every vertex $x$ in
$\mathcal{PC}_{4}$, the size of $f^{-1}(x)$ is not equal to $6$.



We consider two cases:

{\bf Case 1. There is a vertex $\mathbf{x}$ such that
  $\mathbf{|f^{-1}(x)|=7}$.} We may assume that the pre-images of $x$
are exactly the singletons. Then $f$ maps the twenty-one vertices of
size 2 into the five neighbors of $x$, thus there should be a neighbor
$y$ of $x$ which is the image of five such vertices. These five
vertices must share a common element (same arguments as for Corollary
\ref{preImageOf5}). Therefore, we may consider that they are
associated with the sets $\{1,2\}, \{1,3\}, \ldots, \{1,6\}$. By
mapping the empty set and the 2-subset $\{1,7\}$ we still have a
non-surjective homomorphism with no pre-image of size 6 (same
arguments as for Observation \ref{obs:noPreimageOf6}). Therefore, we
may assume that $f^{-1}(x)=N(\emptyset)$ and $f^{-1}(y)=N(\{1\})$.

The remaining 2-subsets (which are the 2-subsets of $[2 \cdot 7]$)
have to be mapped to the four other neighbors of $x$. Among the
3-subsets, the ones containing the element $1$ have to be mapped to
the four other neighbors of $y$. The remaining sets are the 3-subsets
of $[2 \cdot 7]$. In $\mathcal{PC}_{6}$, they induce a matching, each set being
matched to its complement within $[2 \cdot 7]$.

The fifteen 2-subsets of $[2 \cdot 7]$ have to be mapped within the
four neighbors of $x$ which are not $y$. Two such sets can have the
same image only if they share an element. Therefore, the restriction
of $f$ to these vertices induce a coloring of the vertices of
$K(6,2)$. Since $K(6,2)$ is 4-chromatic, the four neighbors of $x$
have a non-empty pre-image. Same argument works for the neighbors of
$y$.


In $\mathcal{PC}_{4}$ there are six vertices which are neither adjacent to $x$
nor to $y$. These six vertices induce a matching in $\mathcal{PC}_{4}$. Each of
the 3-subsets of $[2 \cdot 7]$ has to be mapped simultaneously to a
neighbor of a neighbor of $x$ and a neighbor of a neighbor of
$y$. So these twenty vertices are mapped to the aforementioned six
vertices of $\mathcal{PC}_{4}$. Both sets induce matchings in their respective
graphs, hence if a vertex $a$ is mapped to a vertex $z$, the match of
$a$ has to be mapped to the match of $z$. In other words, if some
vertex $z$ is not in the image of $f$, its match is not either. Since
$f$ is not onto, there must be two such vertices. Thus, all twenty
vertices have to be mapped to four vertices and one of these four
vertices must have a pre-image of size more than 5. By Corollary
\ref{preImageOf5}, its pre-image is included in the neighborhood of
some vertex in $\mathcal{PC}_{6}$. But there is no such 5-tuple among the twenty
considered vertices. This is a contradiction.

%
%

We note that we may actually map the twenty remaining vertices of
$\mathcal{PC}_{6}$ to the six remaining vertices of
$\mathcal{PC}_{4}$, and then obtain a homomorphism of
$\mathcal{PC}_{6}$ into $\mathcal{PC}_{4}$.

{\bf Case 2. For every vertex $\mathbf{x}$ of
  $\mathbf{\mathcal{PC}_{4}}$, $\mathbf{|f^{-1}(x)| \leq 5}$.} In this
case we first note that if $|f^{-1}(x)|=5$ then all five neighbors of $x$
must be in the image of $f$. Otherwise, the twenty-one neighbors of
$f^{-1}(x)$ are mapped to only four vertices and therefore we have a
neighbor $z$ of $x$ with $|f^{-1}(z)| \geq 6$.

Since $f$ is not onto, there is a vertex $z$ in $\mathcal{PC}_{4}$ with an empty
pre-image. Then every neighbor is the image of at most four vertices
from $\mathcal{PC}_{6}$.

{\bf Case 2.1} Suppose there is a neighbor $t$ of $z$ which has a
pre-image of size 4.

Without loss of generality, and using Observation \ref{obs:dist2} and
symmetry arguments, either
$f^{-1}(t)=\left\{\{1\},\{2\},\{3\},\{4\}\right\}$ or
$f^{-1}(t)=\left\{\emptyset,\{1,2\},\{2,3\},\{1,3\}\right\}$.


{\bf Case 2.1.1} If $f^{-1}(t)=\{\emptyset, \{1,2\}, \{1,3\},
\{2,3\}\}$.  Then there are twenty vertices in $N(f^{-1}(t))$ and they
must map to four vertices only. So $N(f^{-1}(t))$ should be
partitioned into four sets of size 5, each part being vertices with a
common neighbor in $\mathcal{PC}_{6}$. But the only vertices having
five neighbors in $N(f^{-1}(t))$ are the vertices associated with the
empty set, $\{1,2\}$, $\{1,3\}$, and $\{2,3\}$. Thus they should be
the center of such partitions, we then denote the corresponding parts
by $P_{\emptyset}, P_{\{1,2\}}, P_{\{1,3\}}$ and
$P_{\{2,3\}}$. Private neighborhoods give us that vertices
$\{4\},\{5\},\{6\}$ and $\{7\}$ are in $P_{\emptyset}$, vertices
$\{1,2,4\},\{1,2,5\},\{1,2,6\}$ and $\{1,2,7\}$ are in $P_{\{1,2\}}$,
vertices $\{1,3,4\},\{1,3,5\},\{1,3,6\}$ and $\{1,3,7\}$ are in
$P_{\{1,3\}}$, and finally vertices $\{2,3,4\},\{2,3,5\},\{2,3,6\}$
and $\{2,3,7\}$ are in $P_{\{2,3\}}$. Moreover, each set then contains
exactly one of the four other vertices in $N(f^{-1}(t))$, i.e.,
$\{1\}, \{2\}, \{3\}, \{1,2,3\}$.

Suppose $x$ is the image of five vertices of $P_{\emptyset}$. By
Observation \ref{obs:5gives5}, there must be a neighbor $y$ of $x$ in
$\mathcal{PC}_{4}$ and a neighbor $a$ of the empty set in
$\mathcal{PC}_{6}$ such that five of the seven neighbors of $a$ are
mapped into $y$, let $N'(a)$ be these five vertices. Note that for
each $b$ in $\{1,2,3\}$, three of the neighbors of $\{b\}$ are already
mapped into $t$, so $a$ cannot be a singleton included in
$\{1,2,3\}$. We may then assume without loss of generality that $a$ is
the singleton $\{4\}$. Then we observe that for any choice of $N'(a)$,
this set $N'(a)$ will have a neighbor in each of the sets
$P_{\emptyset}, P_{\{1,2\}}, P_{\{1,3\}}, P_{\{2,3\}}$. Therefore
vertices $t, y, f(P_{\emptyset}), f(P_{\{1,2\}}), f(P_{\{1,3\}})$ and
$f(P_{\{2,3\}})$ would induce a $K_{2,4}$ in $\mathcal{PC}_{4}$ which
is a contradiction.


{\bf Case 2.1.2} If $f^{-1}(t)=\{\{1\}, \{2\}, \{3\} \{4\}\}$. The set
$f^{-1}(t)$ has nineteen neighbors in $\mathcal{PC}_{6}$ and they should map, by $f$,
to only four neighbors of $t$ in $\mathcal{PC}_{4}$. Thus the neighborhood of
$f^{-1}(t)$ is partitioned into four sets, three of which are of size
5 and the last one of size 4.  The ones of the size 5 must be common
neighbors of a vertex in $\mathcal{PC}_{6}$ and the central vertex itself must
be of the form $\{i\}$, but only one such $i$ can be in $\{5, 6,
7\}$. So without loss of generality we may assume that the first two
parts of size 5 are subsets of $N(\{1\})$ and $N(\{2\})$. Furthermore
since $\{1,2\}$ can only be in one of these two parts, we assume it is
not in the first one. Thus the first part is precisely
$P=\{\{1,3\},\{1,4\},\{1,5\},\{1,6\},\{1,7\}\}$. Let $x$ be the image
of $P$. Then each neighbor of $\emptyset$ and $\{1, 2\}$ except
$\{2\}$ is also a neighbor of a vertex in $P$.  Furthermore,
$f(\{2\})=t$ is also adjacent to $x$. Then if we change $f$ only in
these place, namely defining $f'(\emptyset)=f'(\{1,2\})=v$ and
$f'(a)=f(a)$ otherwise, we will have a have new homomorphism, $f'$,
whose image is a subset of the image of $f$. This new homomorphism
$f'$ would have a vertex with a pre-image of size $7$. But by the Case
1, it is impossible.

{\bf Case 2.2} We finally focus our attention on the case where every
neighbor of $z$ is the image of at most three vertices of
$\mathcal{PC}_{6}$. We remind the reader that we are under the
assumption that pre-image of each vertex has size at most 5. With all
these assumptions we prove the following claim:

\begin{claim}\label{claim:laclaim}
If vertices $x$ and $y$ of $\mathcal{PC}_{4}$ are such that
$|f^{-1}(x)|=|f^{-1}(y)|=5$ and $x$ is adjacent to $y$, then
$f^{-1}(x) \subset N(a)$ and $f^{-1}(x) \subset N(b)$ for some vertices
$a$ and $b$ of $\mathcal{PC}_{6}$ which are adjacent.
\end{claim}

Let $a$ be the common neighbor of the vertices in $f^{-1}(x)$. Note
that $z$ and $x$ are not adjacent and, therefore, have two common
neighbors.  Each of these two common neighbors is the image of at most
three vertices of $\mathcal{PC}_{6}$.  Thus the twenty-one vertices of
$N(f^{-1}(x))$ must be partitioned into five sets three of which are
of size 5 and the other two of size exactly 3. Then $y$ must be the
image of one of the parts of size 5 but these five elements can only
be a common neighbor of vertex $b$ at distance 1 from $a$. This
concludes the proof of Claim \ref{claim:laclaim}.

Having this observed note that there is no vertex mapped to $z$ and
each neighbor of $z$ is the image of at most three vertices, thus at
least forty-nine vertices are mapped to the vertices at distance 2
from $z$.  And, therefore, at least nine of them are the image of five
vertices. These nine vertices induce a subgraph isomorphic to $P^-$,
that is the Petersen graph minus a vertex. Now we consider a mapping
$g$ of $P^-$ which sends each of these nine vertices to the center of
their pre-images under $f$. By Claim \ref{claim:laclaim}, this is a
homomorphism of $P^-$ into $\mathcal{PC}_{6}$. But $P^-$ contains a
$C_5$ while $\mathcal{PC}_{6}$ has odd-girth 7. This contradiction
concludes the proof of Theorem \ref{thm:main}.
\end{proof}

From these results, one can derive the following Corollary.

\begin{corollary}
 Let $G$ be a binary Cayley graph of odd-girth $7$. If $G$ admits a
 homomorphism to $\mathcal{PC}_{4}$, then any such mapping must be onto.
\end{corollary}

\section{Concluding remarks}

Conjecture~\ref{MappingAmongPC2k} can be strengthened in two steps
each of which may give a new idea for proving it. The first
strengthening is based on the following notation:

Given a graph $G$ and a positive integer $l$ we define the
\emph{$l$-th walk power} of $G$, denoted $G^{(l)}$ to be a graph with
vertices set of $G$ as its vertices where two vertices $x$ and $y$
being adjacent if there is a walk of length $l$ connecting $x$ and $y$
in $G$. It follows from this definition that if $\varphi$ is a
homomorphism of $G$ to $H$, then $\varphi$ is also a homomorphism of
$G^{(l)}$ to $H^{(l)}$. Since $\mathcal{PC}_{2k}^{(2k-1)}$ is isomorphic to
$K_{2^{2k}}$, Conjecture~\ref{MappingAmongPC2k} would be implied by
the following conjecture:

\begin{conjecture}
 For $r\geq k$ we have $\chi(\mathcal{PC}_{2r}^{(2k-1)})\geq 2^{2k}$.
\end{conjecture}

It seems then that the methods of algebraic topology used for graph
coloring are the best tools to prove this conjecture. To this end we
suggest the following stronger conjecture, we refer to
\cite{Matousek2003} for definitions and details required for this
conjecture.

\begin{conjecture}
 For $r\geq k$ the simplicial complex associated to $\mathcal{PC}_{2r}^{(2k-1)}$ is
 $2^{2k}$ connected.
\end{conjecture}


Finally, for odd values of $k$ the projective cube $\mathcal{PC}_{k}$
is a bipartite graph and homomorphism problems to or among these
graphs are trivial. However the theory becomes more complicated under
the notion of signed graph homomorphisms and signed projective cubes
as studied in~\cite{NRS13}. Analogue of this work for the case of
singed projective cubes is under development.

\bibliographystyle{plain}
\bibliography{mybib}

\end{document}